\documentclass[12pt]{article}

\usepackage{amssymb}
\usepackage{amsmath, amsthm}
\usepackage{multirow}
\usepackage{hyperref}
\hypersetup{
  colorlinks   = true, 
  urlcolor     = blue, 
  linkcolor    = blue, 
  citecolor   = red 
}

\usepackage{float}
\restylefloat{table}

\let\svthefootnote\thefootnote
\newcommand\freefootnote[1]{%
  \let\thefootnote\relax%
  \footnotetext{#1}%
  \let\thefootnote\svthefootnote%
}

\newtheorem{theorem}{Theorem}[section]
\newtheorem{proposition}[theorem]{Proposition}
\newtheorem{corollary}[theorem]{Corollary}
\newtheorem{lemma}[theorem]{Lemma}

\newtheorem{remark}[theorem]{Remark}

\newcommand{\Z}{\mathbb{Z}}

\newcommand{\R}{\mathbb{R}}
\newcommand{\C}{\mathbb{C}}
\newcommand{\HH}{\mathbb{H}}
\newcommand{\SL}{\operatorname{SL}}
\newcommand{\cc}{\mathbf{c}}
\newcommand{\half}{\tfrac{1}{2}}

\DeclareMathOperator{\ch}{ch}
\DeclareMathOperator{\id}{\mathrm{id}}
\DeclareMathOperator{\tr}{tr}
\DeclareMathOperator{\str}{str}
\DeclareMathOperator{\sgn}{sgn}
\DeclareMathOperator{\wt}{wt}
\DeclareMathOperator{\vac}{\mathbf{1}}

\DeclareMathOperator{\slt}{\mathfrak{sl}_2}

\newcommand{\Vbar}[1]{V_{\bar{#1}}}
\newcommand{ \im}{\textup{i}}

\newcommand{\VLp}{V_{L}}

\newcommand{\thab}[2]{\vartheta 
\begin{bmatrix}
#1\\#2
\end{bmatrix}
}

\newcommand{\textthab}[2]{\vartheta \left[
\begin{smallmatrix}
#1\\#2
\end{smallmatrix}
\right]}

\allowdisplaybreaks[2] 

\begin{document}

\title{Decompositions of index one Jacobi forms into $N=4$ characters and formulas for mock modular forms}
\author{Matthew Krauel\thanks{California State University, Sacramento; e-mail: krauel@csus.edu.
},\quad
 \hfill  Geoffrey Mason\thanks{University of California, Santa Cruz; e-mail: gem@ucsc.edu},\quad
 Michael Tuite\thanks{National University of Ireland, Galway; e-mail: michael.tuite@nuigalway.ie },\quad
 Gaywalee Yamskulna\thanks{Illinois State University; e-mail: gyamsku@ilstu.edu} }
 
 \freefootnote{This project was funded in part by the SQuaREs Program of the American Institute of Mathematics. The first author was also supported by a California State University, Sacramento, RCA Award, and the second author was also supported by the Simons Foundation $\#427007$.}

\date{}
\maketitle

{\abstract
\noindent
It is shown that every weak Jacobi form of weight zero and index one on a congruence subgroup of the full Jacobi group can be decomposed into $N=4$ superconformal characters. Additionally, a simple expression for the mock modular form determining the superconformal character coefficients is obtained, as well as a universal completion structure. Along the way, a useful vector-valued mock modular form is also found and studied.
These results are applied to analyze some Jacobi trace functions associated to super vertex operator algebras and a distinguished sector.
}



\section{Introduction\label{intro}}

The decomposition of weak Jacobi forms on the full Jacobi group into characters of irreducible modules of an $N=4$ superconformal algebra has led to many interesting results and conjectures. Eguchi, Ooguri, and Tachikawa carried this analysis out in sparking the Matheiu moonshine discussions \cite{EOT-Mathieu}.\ Cheng, Duncan, and Harvey \cite{CDH-Umbral, CDH-UmbralNiemeier} also employed this idea and they looked at higher index weak Jacobi forms.\  In both cases, restrictions of the weak Jacobi forms to those related to Niemeier lattices were made. One goal of the present paper is to examine decompositions of certain weak Jacobi forms on subgroups of the Jacobi group.

 Suppose $\phi$ is a weight zero, index one, weak Jacobi form on a subgroup $\Gamma_{\phi}^{J}$ of the full Jacobi group  which decomposes into characters of irreducible unitary $N=4$ Ramond sector modules. That is, we have
 \begin{equation}
 \label{eqn:N=4Decomp}
  \phi  = A_0 \ch_{\frac{1}{4},0} + \sum_{n=0}^\infty B_n \ch_{\frac{1}{4}+n,\half},
 \end{equation}
 where $A_0$ and $B_n$ are constants (the definitions of the  characters $\operatorname{ch}$ and other background material are given in Sections \ref{Section:Background} and \ref{Section:ChiExp} below). One can then consider the coefficient function
 \begin{equation}
\label{eqn:Fdefn}
 F^\phi (\tau) := \sum_{n=0}^\infty B_n q^{n-\frac{1}{8}},
\end{equation}
where $q:=e^{2\pi \im \tau}$ for $\tau \in \mathbb{H}:=\{x+ \im y \in \mathbb{C} \mid y>0\}$. In the event $\Gamma_{\phi}^{J}$ is the full Jacobi group, $\phi$ must be a multiple of the well-known Jacobi form $\phi_{0,1}$ of weight zero and index one. Taking 
$\phi$ to be $2\phi_{0,1}$, the elliptic genus of a $K3$ surface, one obtains
\begin{equation}
\notag 
 F^{K3}(\tau) := F^{2\phi_{0,1}}(\tau) = q^{-\frac{1}{8}}\left(-2+90 q+462 {q}^{2}+1540 {q}^{3}+4554 {q}^{4}+11592 {q}^{5}+\ldots \right).
\end{equation}
Realizing the positive integers in terms of Mathieu group characters goes to the heart of Mathieu moonshine.\ By allowing 
$\phi$ to have higher index and imposing some `extremal' conditions on the Jacobi forms, though still requiring connections between $\Gamma_{\phi}^{J}$ and the Niemeier lattices, Cheng, Harvey and Duncan (loc.\ cit.) also studied functions $F^\phi$  which were central to their Umbral moonshine. In this paper, we develop a simple way to compute and analyze $F^\phi$ for weight zero, index one, $\phi$ and more general $\Gamma_{\phi}^{J}$.

In short, if $\phi $ satisfies \eqref{eqn:N=4Decomp} then we can isolate two intrinsic functions $S_0$ and $S_1$ in the superconformal character decomposition and pair these with the well-known component functions $f_{0} $ and $f_{1} $ of a vector-valued modular form obtained in the theta decomposition \cite[Chapter II, \S 5]{EZ-JacobiForms} of $\phi$. Furthermore, we may reverse this argument to obtain the supercharacter decomposition for any weight zero, index one, weak Jacobi form $\phi$.
 We then obtain our main result.

\begin{theorem}
\label{thm:MainTheoremF}
Let $\phi (z;\tau)$ be a weight zero, index one, weak Jacobi form on some congruence subgroup $\Gamma_{\phi}^{J}:= \Gamma_{\phi}\ltimes \Z^2 $ of the Jacobi group. 
	\begin{enumerate}
		\item[(i)] $\phi$ can be expanded in terms of index one, $N=4$ Ramond superconformal characters, with coefficient function $F^\phi = f_{0} S_1 -f_{1} S_0$. 
		\item[(ii)]  $F^\phi$  is a mock modular form of weight one-half
	(with a multiplier system) for $\Gamma_{\phi}$ with a universal completion
\[
\widehat{F}^\phi (\tau)=F^\phi (\tau)+\chi C(\tau),
\]
where $\chi=\phi(0; \tau)$ is a constant and $C$ is defined in \eqref{eqn:completionC} below.
	\end{enumerate}
\end{theorem}

The functions $S_0$, $S_1$, and $C$ are independent of $\phi$. Thus, $F^\phi$ can be computed directly from its theta decomposition. In fact, we provide some additional tools to compute the components $f_{0} $ and $f_{1} $ in Proposition \ref{prop:EZcoeff} below. To obtain Theorem \ref{thm:MainTheoremF}, the functions $S_0$ and $S_1$ must be found and analyzed.\ They are intimately related to the $\mu$-function considered by Zwegers \cite{Zwegers-Thesis}, which is fundamental to the $N=4$ characters.\ We establish the following.
 
\begin{proposition}
\label{prop:MainProp}
 The functions $S_0$ and $S_1$ are the holomorphic parts of certain indefinite theta functions $\vartheta^0$ and $\vartheta^1$ respectively, and they are both mock modular forms (with a multiplier system) on $\Gamma_0 (4)$.\ Furthermore,
 $(S_0, S_1)$ is a vector-valued mock modular form of weight one on $\operatorname{SL}_2(\mathbb{Z})$.
\end{proposition}

We are also able to use Zwegers' work on the $\mu$-function to discuss the completions of the functions $S_0$ and $S_1$, and thereby the completion of $F^\phi$ also.\ The functions $S_0$ and $S_1$ appear elsewhere in the literature (e.g., in \cite[Display (8.13)]{BFOR-Harmonic}), though to the authors' knowledge, without the descriptions, analysis, and results found here.

As a first application of the theory developed thus far, we analyze the elliptic genus of a $K3$ surface. Unsurprisingly, some of the results we obtain may be found elsewhere in the literature (for example, \cite{EguchiHikami-Mock}). However, as another application, we discuss super vertex operator algebras (SVOAs).

While Monstrous, Mathieu, and Umbral moonshine are ostensibly concerned with the existence and construction of
infinite-dimensional graded modules for finite groups, their identification as vertex operator algebras (VOAs) or SVOAs is 
fundamental.\ As is well-known, the existence of the Monster module $V^\natural$ was nearly obtained by Atkin, Fong, and Smith \cite{AFS-Head}, while its construction as a VOA by Frenkel, Lepowsky, Meurman \cite{FLM} completed the picture (of course, the moonshine-related aspects came later by Borcherds \cite{Borcherds-Monstrous}).\ Gannon proved the existence of the Mathieu module \cite{Gannon-Mathieu}, and his far-reaching method was modified in \cite{DGO-UmbralProof} to establish the existence of the remaining Umbral modules.\ Some of the Umbral moonshine modules have been constructed as twisted modules of SVOAs \cite{ChengDuncan-Meromorphic, DuncanHarvey-UniqueNiemeier, DuncanODesky-SuperVAs}. However, it remains to construct the others, including the elusive Mathieu moonshine module. 

Rather than conjecture the existence of SVOAs or attempt to construct them, we could instead begin with such a structure and ask to what extent it parallels the conjectured Mathieu moonshine SVOA.   In this context, given an arbitrary SVOA $V$, questions arise. For example, (i) does $V$ contain an $N=4$ subalgebra? Additionally, (ii) do the two-variable trace functions associated to the modules for $V$ give rise to mock modular forms? Probing these questions was the motivation for our second example.

Work of the last three authors \cite{MTY-N=4} shed some light on (i).\ Criteria were developed to determine whether a given SVOA contains the $N=4$ superconformal algebra.\ One method, in particular, involved a particular odd lattice of rank six called $L+$ and its associated super lattice theory $V_{L+}$, which indeed contains this superconformal algebra (with central charge six).\ These ideas recently led to a proof \cite{HoehnMason-Most} that almost all super lattice theories
$V_L$ defined by an odd unimodular lattice $L$ of rank $24$ also contain the  superconformal algebra,
arising from an embedding $L+\subseteq L$.\  Theorem \ref{thm:MainTheoremF} provides answers to (ii) in case the relevant trace functions are weak Jacobi forms of weight zero and index one, not necessarily on the full Jacobi group.\ This holds, for example, in the case of $V_{L+}$.\ In this setting
we apply Theorem \ref{thm:MainTheoremF} to the trace function associated to a distinguished twisted module which is a type of `Ramond' sector.

The paper is organized as follows. Section \ref{Section:Background} deals with various automorphic objects.
Section~\ref{Section:ChiExp} is concerned with superconformal character expansions and contains the proof of Theorem~\ref{thm:MainTheoremF}~(i).
Section \ref{Section:S1S2} studies the functions $S_0$ and $S_1$ and contains the proof of Proposition \ref{prop:MainProp}, while Section \ref{Section:FunctionF} includes development of  results about $F^{\phi}$ and the proof of Theorem \ref{thm:MainTheoremF}~(ii). Section \ref{Section:TraceFunctions} discusses trace functions of SVOAs and describes the relevant  parity root elliptic genus.\ The elliptic genus of a $K3$ surface and $V_{L+}$ are discussed in Section \ref{Section:K3} and Section \ref{Section:Rank6}, respectively.

\section{Background on relevant automorphic forms}\label{Section:Background}
For $k $ and $m\in  \mathbb{N}$, a holomorphic function $\phi \colon \mathbb{C} \times \mathbb{H}\to \mathbb{C}$ is a \textit{weak Jacobi form of weight $k$ and index $m$} on $\Gamma_{\phi}\le \SL_{2}(\Z)$, if for all 
 $\gamma =\left(\begin{smallmatrix} a & b \\ c &d \end{smallmatrix}\right) \in \Gamma_{\phi}$ and $(\lambda ,\mu)\in \Z^{2}$ we have 
\begin{align}
\notag
 \phi \left( \frac{z}{c\tau +d};\frac{a\tau +b}{c\tau +d} \right) &=  (c\tau +d)^k e^{2\pi \im m \left(\frac{cz^2}{c\tau +d}\right)}\phi (z; \tau),
 \\ 
 \notag
 \phi (z +\lambda \tau +\mu ; \tau ) &= e^{-2\pi \im m \left(\lambda^2 \tau +2\lambda z\right)} \phi (z; \tau),
\end{align}
and $\phi$ has an expansion of the form
\begin{equation}
\notag
 (c\tau +d)^{-k} e^{-2\pi \im m \left(\frac{cz^2}{c\tau +d}\right)}\phi \left( \frac{z}{c\tau +d};\frac{a\tau +b}{c\tau +d} \right) =  \sum_{\substack{n,r\in  \mathbb{Z} \\ n\geq 0}} c(n,r) q^{n} \zeta^r
\end{equation}
for any $\gamma \in \operatorname{SL}_2(\mathbb{Z})$, where $q=e^{2\pi\im\tau}$ and $\zeta=e^{2\pi\im z}$. The elements $(\gamma ,(\lambda,\mu))$ generate $\Gamma_{\phi}^{J}:= \Gamma_{\phi}\ltimes \Z^2 $ which is a subgroup of the full  Jacobi group $\Gamma^J=\SL_{2}(\Z)\ltimes \Z^2 $ (e.g.\ \cite{EZ-JacobiForms, BFOR-Harmonic}). 
If $\phi$ satisfies the stronger condition that $c(n,r)=0$ when $4mn<r^2$, then it is simply called a \textit{Jacobi form}.  

Let $\Gamma (M)$ be the principal congruence subgroup of $\SL_{2}(\Z)$ defined by
\begin{align}
\notag 
\Gamma (M):=\left\{\begin{pmatrix}
a&b\\c&d
\end{pmatrix}\in \operatorname{SL}_2(\mathbb{Z}) \, \, \biggl \vert \,\, \begin{pmatrix}
a&b\\c&d
\end{pmatrix} 
\equiv 
\begin{pmatrix}
1&0\\0 &1
\end{pmatrix} 
\bmod M \right\}.
\end{align}
Then we say that $\phi$ is a Jacobi form  of level $M$ if $\phi$ satisfies the above definition of a Jacobi form  and $M$ is the least positive integer such that  $\Gamma(M)\le \Gamma_{\phi}$.

Jacobi forms are related to classical Jacobi theta series with characteristic\footnote{In \cite{FK-ThetaConstants} $\theta$ is used in place of $\vartheta$ and their characteristic notation differs from ours by a factor of $\frac{1}{2}$.} $a,b\in\R$ defined by (see, for example, \cite[Definition 1.1]{FK-ThetaConstants})
\begin{align}
\thab{a}{b}(z; \tau)&:
=\sum_{n\in \mathbb{Z}}e^{2\pi \im (n+a)b}q^{\frac{1}{2}(n+a)^2}\zeta ^{n+a}
.
\notag 
\end{align}
These functions satisfy
\begin{align}
\thab{a+1}{b}
(z; \tau )
&=
\thab{a}{b}
(z; \tau ),
\label{eq:aplus1}
\\
\thab{a}{b+1}
(z; \tau )
&=e^{2\pi \im a}
\thab{a}{b}
(z; \tau ),
\notag
\\
\thab{-a}{-b}
(-z; \tau )
&=
\thab{a}{b}
(z; \tau ).
\label{eq:minus}
\end{align}
Additionally, for all $\lambda,\mu  \in \R$ we have the ``continuous spectral flow'' 
\begin{align}
\thab{a}{b}(z+\lambda  \tau+\mu; \tau)&=
e^{-2\pi \im \lambda (b+\mu)}q^{-\frac{1}{2}\lambda ^2}\zeta^{-\lambda } 
\thab{a+\lambda }{b+\mu}(z; \tau),
\label{eq:thetaflow}
\end{align}
and in particular,  for all $\lambda,\mu \in\Z$ we find 
\begin{align}
\thab{a}{b}(z+\lambda  \tau+\mu; \tau) &=
e^{2\pi \im (\mu a-\lambda  b)}q^{-\frac{1}{2}\lambda ^2} \zeta ^{-\lambda }
\thab{a}{b} (z; \tau ). 
\notag
\end{align} 
We also note the modular transformation properties under the action of the
standard $\SL_2(\Z)$ generators $S=
\left(
\begin{smallmatrix}
0 & 1 \\ 
-1 & 0
\end{smallmatrix}\right) $ and $T= \left(
\begin{smallmatrix}
1 & 1 \\ 
0 & 1
\end{smallmatrix}\right) $ (with relations $(ST)^{3}=-S^{2}=I$) 
\begin{equation}
\label{eqn:ThetaTransforms}
\begin{aligned}
\thab{a}{b}(z; \tau +1) &=e^{- \im \pi a(a+1)}
\thab{a}{b+a+\frac{1}{2}}
(z; \tau ),  
  \\
\thab{a}{b} \left(\frac{z}{\tau }; -\frac{1}{\tau }\right) &=
(- \im \tau )^{\frac{1}{2}}e^{2\pi iab}e^{ \im\pi \frac{z^{2}}{ \tau} }\thab{b}{-a}(z; \tau ).  
\end{aligned}
\end{equation}
Special cases of Jacobi theta series include  $\theta_j(z; \tau)$ defined for $j=1,2,3,4$ by
\begin{equation}
\notag 
 \theta_1  := -\thab{\frac{1}{2}}{\frac{1}{2}}, \quad \theta_2 := \thab{\frac{1}{2}}{0}, \quad \theta_3 := \thab{0}{0}, \quad \text{and} \quad \theta_4 := \thab{0}{\frac{1}{2}},
\end{equation}
(the notation differs throughout the literature), which are Jacobi forms\footnote{\label{note} See \cite{CG-JacobiThetaSeries}  for details about Jacobi forms of  half-integral index or weight and with a character.} of weight and index $\frac{1}{2}$, and level $2$ with a character (or multiplier system). In each of these cases, we also need the theta-constants defined by
\begin{equation}
\notag 
 \theta_j (\tau) := \theta_j(0; \tau).
\end{equation}

The theta decomposition of a Jacobi form of index $m$ employs the  theta series 
\begin{align}
\theta_{m,\mu}(z; \tau)&:= \sum_{\ell =\mu \bmod{2m}}q^{\frac{\ell^2}{4m}}\zeta^{\ell} = \thab{\frac{\mu}{2m}}{0}(2mz; 2m\tau),
\label{eq:EZtheta}
\end{align}
for $ \mu\in \Z_{2m}$ \cite[Section 5]{EZ-JacobiForms}.
Following  \eqref{eq:thetaflow}, these satisfy the ``spectral flow'' relations
\begin{equation}
\label{eq:EZthetaflow}
\begin{aligned}
\theta_{m,\mu}\left(z+\tfrac{1}{2}; \tau\right) &= e^{\pi \im \mu} \theta_{m,\mu}(z; \tau)
\\ \theta_{m,\mu}\left(z+\tfrac{\tau}{2}; \tau\right) &= q^{-\tfrac{m}{4}}\zeta^{-m} \theta_{m,\mu +m}(z; \tau).
\end{aligned}
\end{equation}
 We also note that
\begin{equation}
\label{eqn:ThetaRelation}
\theta_{1,0}(z; \tau)=\theta_{3}(2z; 2\tau)  \quad \text{and} \quad \theta_{1,1}(z; \tau)= \theta_{2}(2z; 2\tau).
\end{equation}
	
All of these theta series are special cases of theta series attached to lattices. Let $Q$ be a positive definite rational-valued quadratic form on a lattice $L$ with  $\operatorname{rank}(L)$ even, and  $B(\alpha ,\beta )=Q(\alpha +\beta)-Q(\alpha)-Q(\beta)$ the associated bilinear form. For fixed $h\in L$ we find 
\begin{equation}
\label{eq:thetaLh}
 \theta_{L}^h (z; \tau) := \sum_{\alpha \in L} q^{Q(\alpha)} \zeta^{B(\alpha ,h)},
\end{equation}
is a Jacobi form (cf.\ footnote~\ref{note}) with a character of weight $\operatorname{rank}(L)/2$, index $Q(h)$, and level $M$ equal to the level of $Q$. More specifically, if $A$ is the matrix representing $Q$, then $M$ is the smallest positive integer such that $M(A^{-1})$ has integral entries. 

We will also be utilizing the $\mu$-function of Zwegers \cite[Proposition 1.4]{Zwegers-Thesis} defined for $z_1,z_2\in \mathbb{C}\setminus (\mathbb{Z}\tau +\mathbb{Z})$ and $\tau \in \mathbb{H}$ by
\begin{equation}
 \notag 
 \mu \left(z_1,z_2;\tau\right) := \im \frac{e^{\pi \im z_1}}{\theta_1 \left(z; \tau_2\right)} \sum_{n\in \mathbb{Z}} \frac{(-1)^n q^{\frac{1}{2}n(n+1)} e^{2\pi \im n z_2}}{1-e^{2\pi \im z_1}q^n}.
\end{equation}
In particular, we need\footnote{Our definition here is precisely as in \cite{EguchiHikami-Mock}, but is equal to $ \im \mu$ for the $\mu$ in \cite{Zwegers-Thesis}. }
\begin{equation}
 \label{eqn:Bigmu}
 \mu \left(z; \tau\right) := \mu \left(z,z;\tau \right) = \im \frac{\zeta^{\frac{1}{2}}}{\theta_1 (z; \tau)} \sum_{n\in \mathbb{Z}} \frac{(-1)^n q^{\frac{1}{2}n(n+1)} \zeta^n}{1-\zeta q^n}.
\end{equation}
We note that although $\mu \left(z; \tau\right)$ is an elliptic function  
\begin{align}
	\label{eqn:muElliptic}
\mu(z+\lambda \tau+\mu; \tau)= \mu(z; \tau),\quad 	(\lambda,\mu)\in\Z^{2},
\end{align}
it is a weight $\half $, index $0$, mock Jacobi form for $\Gamma^{J}$ with a non-holomorphic completion (see \cite[Proposition 1.11]{Zwegers-Thesis} or \cite{EguchiHikami-Mock} for details)
\begin{align*}
\widehat{\mu}(z; \tau):=\mu(z; \tau)-C(\tau),
\end{align*} 
where we define
\begin{equation}
\label{eqn:completionC}
C(\tau):=\half\sum_{m\in\Z}(-1)^{m}\sgn \left(m+\half \right)\beta\left(2\left(m+\half \right)^2 y\right)q^{-\half (m+\half )^2}.
\end{equation}
Here, $\operatorname{sgn}$ is the typical signature function whose output is $1,0,-1$, depending on whether the input is positive, 0, or negative, respectively, while $y:=\operatorname{im}(\tau)$ and $\beta(x)=\int_{x}^{\infty}u^{-\half}e^{-\pi u}\, du$ for real $x\ge 0$.
 Note that $C(\tau)=\tfrac{ \im}{2} R(0;\tau)$ for a more general function $R(z;\tau)$ of \cite[below Lemma 1.7]{Zwegers-Thesis} and $C(\tau)= \half R(0;\tau)$ in \cite[(3.8)]{EguchiHikami-Mock}  (and is not to be confused with $R$ of \eqref{eq:Rchar} below). Then one finds from \cite[Proposition 1.11]{Zwegers-Thesis} (see also \cite[(3.14)]{EguchiHikami-Mock}) that
\begin{align}
\widehat{\mu}(z; \tau+1)&=e^{- \tfrac{\im \pi}{4}}\widehat{\mu}(z; \tau),
\notag
\\
\widehat{\mu}\left(\frac{z}{\tau}; -\frac{1}{\tau}\right)&=-(- \im\tau)^{\half}\widehat{\mu}(z; \tau),
\notag
\\
\widehat{\mu}(z+\lambda \tau+\mu; \tau)&=\widehat{\mu}(z; \tau),\quad 	(\lambda,\mu)\in\Z^{2}.
\notag
\end{align}

\section{Superconformal character expansions}\label{Section:ChiExp}
\subsection{Superconformal algebra}
Let $\mathcal{A}$ denote the $N=4$ superconformal algebra of central charge $6$, which contains subalgebras isomorphic to the affine Lie algebra $\widehat{\mathfrak{sl}}_2$ of level $1$ and the Virasoro algebra of central charge $6$ (see, for example, \cite{MTY-N=4}). One can isolate elements $\omega$ and $h$ of $\mathcal{A}$ so that their associated zero-mode endomorphisms $\frac{1}{2}h(0)$ and $L(0):=\omega(1)$ on the unitary irreducible highest weight Neveu-Schwarz (NS) and Ramond (R) representations $V^{NS}_{\alpha ,\beta}$ and $V^{R}_{\alpha ,\beta}$
of $\mathcal{A}$ have eigenvalues $\alpha$ and $\beta$, respectively.\footnote{The element $h$ here is often denoted by $J^3$ elsewhere in the literature.} 
We are able to normalize $h$ so that $h(0)$ has integral eigenvalues on each $V^{R}_{\alpha ,\beta}$ (though the factor of $\frac{1}{2}$ on $h(0)$ above means $\beta \in \frac{1}{2}\mathbb{Z}$).

The Ramond characters  in this setting are given by
\begin{equation}
\notag
 \operatorname{ch}_{\alpha , \beta} (z; \tau) = \operatorname{tr}_{V^{R}_{\alpha ,\beta}} \left((-1)^{h(0)} \zeta^{h(0)} q^{L(0)-\frac{1}{4}}\right).
\end{equation}
The form of $\operatorname{ch}_{\alpha , \beta}$ depends greatly on whether the representation $V^{R}_{\alpha ,\beta}$ is massive (non-BPS, non-supersymmetric, long) or massless (BPS, supersymmetric, short). The massive representations are indexed by $\alpha \in \frac{1}{4} +\ell  $ for $\ell \in \mathbb{N}$ and $\beta =\frac{1}{2}$ with characters  \cite{ET-N4characters}
\begin{equation}
\label{eqn:N=4long}
 \operatorname{ch}_{\frac{1}{4}+\ell ,\frac{1}{2}} (z; \tau) = q^{\ell -\frac{1}{8}} \frac{\theta_1 (z; \tau)^2}{\eta (\tau)^3}.
\end{equation}
The massless representations are indexed by $\alpha =\frac{1}{4}$ and $\beta =\frac{j}{2}$ for $j\in \{0,1\}$, and their characters can be given by
\begin{equation}
 \notag 
 \operatorname{ch}_{\frac{1}{4} ,0} (z; \tau) = \frac{\theta_1 (z; \tau)^2}{\eta (\tau)^3}\mu (z; \tau),
\end{equation}
for $\mu(z; \tau)$ of \eqref{eqn:Bigmu}, and (using a common relation, for example, \cite[(1.10)]{EOT-Mathieu})
\begin{equation}
\label{eqn:N=4short1}
 \operatorname{ch}_{\frac{1}{4} ,\frac{1}{2}} (z; \tau) = q^{-\frac{1}{8}} \frac{\theta_1 (z; \tau)^2}{\eta (\tau)^3} -2\frac{\theta_1 (z; \tau)^2}{\eta (\tau)^3}\mu (z; \tau).
\end{equation}

We are interested in decomposing weak Jacobi forms into the characters above. In fact, this idea was employed in \cite{CDH-Umbral,CDH-UmbralNiemeier} and leads to the definition of extremal Jacobi forms. The analysis there includes more general higher index Jacobi forms stemming from Niemeier lattices. The work surrounding Umbral moonshine also takes a slightly different path than that taken here. In particular,  \textit{loc.\ cit.} decomposes a meromorphic Jacobi form obtained from a Jacobi form $\phi$ as discussed by Zwegers \cite{Zwegers-Thesis}.

\subsection{Character expansion of $\phi$ }
Let $\phi(z; \tau)$ be a weak Jacobi form 
of weight $0$ and index $1$ for some of level $M$ subgroup $\Gamma_{\phi}^{J}$ of the Jacobi group. As such, $\phi (0; \tau)$ is a holomorphic modular form of weight $0$ and therefore constant.
In order to prove  Theorem~\ref{thm:MainTheoremF}, let us first assume that 
we may decompose $\phi$ into index $1$  Ramond characters of the $N=4$ superconformal algebra, i.e., there exist constants $A_0$, $B_n$ such that
\begin{align*}
 \phi(z; \tau)  &
 = A_0 \ch_{\frac{1}{4},0}(z; \tau) + B_0 \ch_{\frac{1}{4},\frac{1}{2}}(z; \tau) + \sum_{n=1}^\infty B_n \ch_{\frac{1}{4}+n,\frac{1}{2}}(z; \tau).
\end{align*}
Thus, using \eqref{eqn:N=4long}--\eqref{eqn:N=4short1} we obtain
\begin{align} 
\phi (z; \tau)  &= \left(A_0 -2B_0\right) \ch_{\frac{1}{4},0}(z; \tau) + B_0 q^{-\frac{1}{8}} \frac{\theta_1 (z; \tau)^{2}}{\eta (\tau)^{3}} + \sum_{n=1}^\infty B_n q^{n -\frac{1}{8}} \frac{\theta_1(z; \tau)^{2}}{\eta (\tau)^{3} }
\notag
\\&= \left(A_0 -2B_0\right)  \frac{\theta_1 (z; \tau)^{2}}{\eta (\tau)^{3}} \mu (z; \tau)  + \sum_{n=0}^\infty B_n q^{n -\frac{1}{8}} \frac{\theta_1(z; \tau)^{2}}{\eta (\tau)^{3} }.
\label{eqn:N=4Decomp01}
\end{align}
Reorganizing we obtain
\begin{align}
\eta (\tau)^{3}\phi(z; \tau)=  \chi\,\theta_1 (z; \tau)^{2} \mu (z; \tau) 
+F^{\phi}(\tau)\theta_1 (z; \tau)^{2},
\label{eq:phi_f}
\end{align}
where 
\begin{align}
\label{eq:f_chi_def}
\chi :=\phi(0; \tau)=A_{0} -2B_{0},
\end{align}
 (since $\ch_{\frac{1}{4},0}(0; \tau)=1$) and for  $F^{\phi}$ of  \eqref{eqn:Fdefn}.

\begin{remark}
Taking $\phi = 2\phi_{0,1}$ in \eqref{eqn:N=4Decomp01} gives the well-known decomposition of  the $K3$ elliptic genus \cite{EOT-Mathieu} with $\chi^{K3}=24$ and
\begin{align*}
 F^{K3}(\tau) &= 2q^{-\frac{1}{8}} \left( -1 +45q +231q^{2} +770q^{3} +2277q^{4} +\cdots \right)
\\ &= q^{-\frac{1}{8}} \left( B_0 + B_1q +B_2q^{2} +B_3q^{3} +B_4q^{4} +\cdots \right).
\end{align*}
In the literature, it is common for $F^{K3}$ to be denoted by $H^{(2)}$ or $h^{(2)}$, and this is the mock modular object that interested \cite{EOT-Mathieu}, and it is one of the objects of interest in \cite{CDH-Umbral,CDH-UmbralNiemeier} as well.
\end{remark}

In order to study $F^{\phi}$ 
we analyze the theta decomposition of \eqref{eq:phi_f}  in terms of $\theta_{1,\mu}$ of \eqref{eq:EZtheta} (cf.\ \cite[Section~5]{EZ-JacobiForms}).
The decomposition of the $F^{\phi}(\tau)\theta_1 (z; \tau)^{2}$ term follows from the identity\footnote{This can be shown by noting the ratio of each side of \eqref{eq:theta1sqid} is a holomorphic weight $0$, index $0$, Jacobi form and is therefore a constant which evaluates to unity.} 
\begin{align}
	\theta_{1}(z; \tau)^{2}=\theta_{2}(2\tau)\theta_{3}(2z; 2\tau)-\theta_{3}(2\tau)\theta_{2}(2z; 2\tau),
	\label{eq:theta1sqid}
\end{align}
using \eqref{eqn:ThetaRelation}. Following \eqref{eqn:muElliptic}, we similarly find that the $\theta_1 (z; \tau)^{2} \mu (z; \tau)$ term has the decomposition
\begin{align}
\label{eq:theta1_mu}
\theta_1 (z; \tau)^{2} \mu (z; \tau)
=S_{0}(\tau)\theta_{3}(2z; 2\tau)+S_{1}(\tau)\theta_{2}(2z; 2\tau),
\end{align}
 where here the coefficient functions $S_{r}(\tau)$ for $r=\ell \bmod 2$ are the $\zeta$-coefficients
\begin{align}
	\label{eq:Sdef}
	S_{\ell}(\tau):=[q^{\frac{1}{4}\ell^{2}}\zeta^{\ell}]\left(\theta_1 (z; \tau)^{2} \mu (z; \tau) \right).
\end{align}
 Since  $\ch_{\frac{1}{4},0}^{(1)}(0; \tau)=1$ we also obtain from \eqref{eq:phi_f} that  
\begin{align}
S_{0}(\tau)\theta_3(2\tau)+S_{1}(\tau)\theta_2(2\tau)=\eta(\tau)^{3}.
\label{eq:S0S1eta}
\end{align}
Explicit expressions for $S_{1},S_{2}$ and their mock modular properties are explored in the next section.
Lastly, the theta decomposition of $\phi(z; \tau)$ is
\begin{align}
	\label{eqn:EZdecomp}
	\phi(z; \tau)=& \sum_{n=0}^\infty \sum_{r\in  \mathbb{Z}} c(n,r) q^{n} \zeta^r
	= f_{0}(\tau)\theta_{3}(2z; 2\tau)+f_{1} (\tau)\theta_{2}(2z; 2\tau),
\end{align}
where for $r=0,1$, we define 
\begin{align*}
	f_{r}(\tau):=f_{r}^{\phi}(\tau):=\sum_{m =0}^\infty a_{r}(m)q^{\tfrac{m}{4}},\quad
	a_{r}(m):=c\left( \tfrac{m+r^{2}}{4},r\right).
\end{align*}
From \eqref{eq:f_chi_def} we find
\begin{align}
	f_{0}(\tau)\theta_{3}(2\tau)+f_{1} (\tau)\theta_{2}(2\tau)=\chi.
	\label{eq:f01chi}
\end{align}
Together, \eqref{eq:phi_f}, \eqref{eq:theta1sqid}, \eqref{eq:S0S1eta},  and \eqref{eqn:EZdecomp} imply
\begin{align}
	\eta(\tau)^{3}f_{0}(\tau)&=\chi  S_{0}(\tau)+\theta_2(2\tau) F^{\phi}(\tau),
	\label{eq:hS0f}
	\\
	\eta(\tau)^{3}f_{1}(\tau)&=\chi  S_{1}(\tau)-\theta_3(2\tau) F^{\phi}(\tau).
	\label{eq:hS1f}
\end{align}
Thus, $F^{\phi}$ is determined by the theta coefficient functions $f_{r}$ and the $S_{r}$ as follows.
\begin{proposition}
	\label{prop:Fhr}
	We have
	\begin{align}
		F^{\phi}(\tau)=f_{0}(\tau)S_{1}(\tau)-f_{1}(\tau)S_{0}(\tau).
		\label{eq:Fphi}
	\end{align}
\end{proposition}
\begin{proof}
	Combine \eqref{eq:S0S1eta} and \eqref{eq:f01chi} into the matrix equation
	\[
	M\begin{pmatrix}
		\theta_{3}(2\tau) \\ \theta_{2}(2\tau)
	\end{pmatrix}
	=
	\begin{pmatrix}
		\chi  \\ \eta(\tau)^{3}
	\end{pmatrix},
	\quad
	M:=\begin{pmatrix}
		f_{0}(\tau) & f_{1}(\tau)
		\\
		S_{0}(\tau) & S_{1}(\tau)
	\end{pmatrix}.
	\]
	Hence we have
	\begin{align*}
		\begin{pmatrix}
			\theta_{3}(2\tau) \\ \theta_{2}(2\tau)
		\end{pmatrix}
		&= \frac{1}{\det (M) }
		\begin{pmatrix}
			S_{1}(\tau) & -f_{1}(\tau)
			\\
			-S_{0}(\tau) & f_{0}(\tau)
		\end{pmatrix}
		\begin{pmatrix}
			\chi  \\ \eta(\tau)^{3}
		\end{pmatrix}
		= \frac{F^{\phi}(\tau)}{\det (M)}\begin{pmatrix}
			\theta_{3}(2\tau) \\ \theta_{2}(2\tau)
		\end{pmatrix},
	\end{align*}
	from \eqref{eq:hS0f} and \eqref{eq:hS1f}. Thus $F^{\phi}(\tau)=\det (M)$ and the result follows.
\end{proof}
Conversely, for any weight $0$, index $1$, holomorphic weak Jacobi form $\phi$, we may define $F^{\phi}$ by \eqref{eq:Fphi} and then confirm \eqref{eq:phi_f} (using \eqref{eq:theta1sqid}--\eqref{eq:f01chi}), thus showing that $\phi$ has a superconformal character expansion with coefficient function $F^\phi$. This proves (i) of Theorem~\ref{thm:MainTheoremF}.


\section{Two intrinsic mock modular forms}\label{Section:S1S2}
To find an explicit expression for $S_{\ell}$ of \eqref{eq:Sdef} we note that
\begin{align*}
\theta_1 (z; \tau)^{2}\mu (z; \tau) 
&= \im\, \zeta^{\half}\theta_1 (z; \tau) \sum_{n\in \Z} (-1)^n \frac{q^{\half n(n+1)}\zeta^n}{1-q^n\zeta}.
\end{align*}
For $\lvert q^n \zeta \rvert <1$ and $\lvert q^n \zeta^{-1} \rvert <1$, manipulating the series we find
\begin{align*}
 &\sum_{n\in \Z} (-1)^n \frac{q^{\half n(n+1)}\zeta^n}{1-q^n\zeta} 
= \left(\sum_{m,n \geq 0} - \sum_{m,n <0}\right) (-1)^n q^{\half \left(n^{2}+2mn+n\right)} \zeta^{m+n}.
\end{align*}
We may also write 
\begin{align*}
 \im\, \zeta^{\half} \theta_1(z; \tau)  
 = \sum_{k\in \Z} (-1)^k q^{\half \left(k+\half \right)^{2}} \zeta^{-k},
\end{align*}
so that
\begin{align*}
 \theta_1 (z; \tau)^{2} \mu (z; \tau) 
 = \left[ \left(\sum_{m,n \geq 0} - \sum_{m,n <0}\right) (-1)^n q^{\half \left(n^{2}+2mn+n\right)} \zeta^{m+n} \right] \left[ \sum_{k\in \Z} (-1)^k q^{\half \left(k+\half \right)^{2}} \zeta^{-k} \right].
\end{align*}
From \eqref{eq:Sdef} we therefore find
\begin{align}
	\notag
S_\ell (\tau) =&q^{-\frac{1}{4}\ell^{2}} \left(\sum_{m,n \geq 0} - \sum_{m,n <0}\right) (-1)^{m+\ell} q^{\half \left(n^{2}+2mn +n\right)}q^{\half  \left(m+n-\ell +\half  \right)^{2}  }
\\
\label{eq:Ssum}
 =&\left(\sum_{m,n \geq 0} - \sum_{m,n <0}\right) (-1)^{m+\ell} q^{ \half  \left(m +2n-\ell + \half \right)^{2}-\left(n-\half \ell\right)^{2}}.
\end{align} 
One can check that $S_{\ell}(\tau)=S_{\ell+2}(\tau)$  as expected. 

\begin{remark}  \
\begin{enumerate}
 \item[(a)] The initial parts of the $q$-series are given by
\begin{align*}
S_{0}(\tau)=& q^{\frac{1}{8}}\left( 
1-q+2\,q^{2}+q^{3}-2\,q^{5}+q^{6}+2\,q^{9}+q^{10}
+\ldots\right),
\\
S_{1}(\tau)=&2\,q^{-\frac{1}{8}}\left( 
-q+q^{3}-q^{4}-q^{6}+q^{8}-q^{9}+q^{10}
+\ldots\right).
\end{align*}
Below we show that $S_{0},S_{1}$ define a weight $1$ vector-valued mock modular form for $\SL_2(\Z)$. 
 \item[(b)] The functions $S_{0},S_{1}$ occur elsewhere in the literature. For example, $S_{0}=-F_{+}(0)$ and  $S_{1}=-F_{-}(0)$ for $F_{\pm}(z)$ functions of \cite[Display (8.13)]{BFOR-Harmonic} used in the Zwegers' analysis of $\mu(z_1,z_2;\tau)$.
Likewise the relation \eqref{eq:S0S1eta} is that of \cite[Display (8.11)]{BFOR-Harmonic}.
  See Lemma~\ref{lem:SF} below as well.
\end{enumerate}
\end{remark}
We now relate $S_\ell$ to an indefinite theta function for the rank $2$ quadratic form
\begin{align}
\label{eqn:QABdef}
Q(\alpha) :=\half \alpha^T A \alpha,\quad 
\quad A :=\begin{pmatrix} 1 & 2 \\ 2 & 2\end{pmatrix},
\end{align}
with associated bilinear form $B(\alpha ,\beta ):= \alpha^T A \beta$, where $\alpha ,\beta \in \mathbb{Q}^2$ and $T$ denotes the transpose.
\begin{lemma}
	\label{lem:Selltheta}
For $\ell \in\Z$ we have
\[
S_\ell (\tau ) = (-1)^{\ell +1}\frac{ \im}{2}  \sum_{\nu \in a(\ell)+\Z^{2}} \left(\sgn B(\nu ,c_1)-\sgn B(\nu ,c_2)\right) e^{2\pi \im B(\nu ,b)} q^{Q(\nu)},
\] 
where 
\begin{align}
 a(\ell) :=\begin{pmatrix} \half  \\- \half \ell \end{pmatrix}, \quad 
 b :=\begin{pmatrix} \half  \\ 0\end{pmatrix}, \quad 
 c_1  :=\begin{pmatrix} -1 \\ 1\end{pmatrix}, \quad 
  c_2  :=\begin{pmatrix} -2 \\ 1\end{pmatrix}.
  \label{eq:abc_data}
\end{align}
\end{lemma}
\begin{proof}
First note that for $\nu 
=\left(\begin{smallmatrix}
m+ \frac{1}{2}  \\ n- \frac{1}{2} \ell
\end{smallmatrix}\right)\in a(\ell)+\Z^{2}$ we have
\begin{align}
\label{eq:Qnumn}
 Q(m,n): =Q(\nu) = \half  \left(m +2n-\ell + \half \right)^{2}-\left(n-\half \ell\right)^{2}.
\end{align}  
Furthermore, $2Q(c_1)=Q(c_2)=-1$ and $B(c_1,c_2)=-2<0$. Since $B(\nu ,b)=\tfrac{1}{2}(m-\ell)+\tfrac{1}{4}+n$ we have 
 \[
 e^{2\pi \im B(\nu,b)}  = \im\, (-1)^{m-\ell}.
 \]
  Meanwhile, using that $B(\nu ,c_1)=m+\tfrac{1}{2}$ and $B(\mu ,c_2)=-2n+\ell$ we have 
 \[
 \sgn \left(B\left(\nu ,c_1 \right)\right) =\begin{cases} 
 1 & \text{ if }m \geq 0 \\ -1 & \text{ if }m < 0 
 \end{cases}
\quad  \mbox{ and } \quad
 -\sgn \left(B\left(\nu ,c_2 \right)\right) =\begin{cases} 
 1 & \text{ if }n> \half \ell \\ 0 &\text{ if } n=\half \ell \\ -1 & \text{ if }n< \half \ell 
 \end{cases}
 \]
where we may assume that $\ell\ge 0$ without loss of generality.
 Thus,
 \[
 \sgn \left(B\left(\nu ,c_1 \right)\right) -\sgn \left(B\left(\nu ,c_2 \right)\right)
 =\begin{cases}
 \;\;2 &\text{ if } m\geq 0, n>\half \ell
 \\ 
 \;\; 1 &\text{ if } m\geq 0, n=\half \ell
 \\
 \;\; 0 &\text{ if } m\geq 0, n<\half \ell \;\text{ or } m<0 , n>\half \ell
 \\ -1 &\text{ if } m< 0, n=\half \ell
 \\ -2 &\text{ if } m< 0, n<\half \ell .
 \end{cases}
 \]
We find
  \begin{align*}
 &  (-1)^{\ell +1} \frac{ \im}{2} \sum_{\nu \in a+\Z^{2}} \left(\sgn B(\nu ,c_1)-\sgn B(\nu ,c_2)\right) e^{2\pi \im B(\nu ,b)} q^{Q(\nu)}
 \\&= \left(\sum_{\substack{m\geq 0 \\ n>\ell /2}} - \sum_{\substack{m< 0 \\ n<\ell /2}}\right) (-1)^m q^{Q(m ,n)}
 + \half \left(\sum_{m\geq 0} - \sum_{m<0}\right) (-1)^m q^{Q\left(m ,\half \ell\right)}
 \\&= \left(\sum_{\substack{m\geq 0 \\ n\geq 0}} - \sum_{\substack{m< 0 \\ n<0}}\right) (-1)^m q^{Q(m ,n)}
 - \sum_{\substack{m\geq 0 \\ 0\leq n \leq \ell /2}} (-1)^m q^{Q(m ,n)}
 - \sum_{\substack{m<0 \\ 0\leq n < \ell /2}} (-1)^m q^{Q(m ,n)}
 \\ &\qquad + \half \left(\sum_{m\geq 0} - \sum_{m<0}\right) (-1)^m q^{Q\left(m ,\half \ell\right)}
 \\&= S_\ell (\tau)
 - \sum_{\substack{m\geq 0 \\ 0\leq n \leq \ell /2}} (-1)^m q^{Q(m ,n)}
 - \sum_{\substack{m<0 \\ 0\leq n \leq \ell /2}} (-1)^m q^{Q(m ,n)}
 + \half \left(\sum_{m\geq 0} + \sum_{m<0}\right) (-1)^m q^{Q\left(m ,\half \ell\right)}
 \\&= S_\ell (\tau)
 - \sum_{\substack{m\in \Z \\ 0\leq n \leq \ell /2}} (-1)^m q^{Q(m ,n)}
 + \half \sum_{m\in \Z} (-1)^m q^{Q\left(m ,\half \ell\right)}.
 \end{align*}
 Note that from \eqref{eq:Qnumn} it follows that $Q(m',n)=Q(m,n)$ for $ m'=-m-1-4n+2\ell$. Thus for fixed $n$ we find
 \begin{align*}
\sum_{m\in\Z}(-1)^{m} q^{Q(m,n)} =-\sum_{m'\in\Z}(-1)^{m'} q^{Q(m',n)}=0.
 \end{align*}
 Therefore the result holds.
 \end{proof}

Following \cite[(8.13)]{BFOR-Harmonic}, the series $S_{\ell}$ may also be expressed as a nested double sum.
\begin{lemma}
\label{lem:SF}
We have  
\[
S_{\ell}(\tau)=\sum_{s=\ell  \bmod{2}}\;\sum_{r\ge s}(-1)^{r}q^{P(r,s)},
\]
for $P(r,s):=\half (r+\half )^2-\frac{1}{4}s^2$.
\end{lemma}

\begin{proof}
Let $s=2n-\ell$ and $r=m+s=m+2n-\ell$  in \eqref{eq:Ssum}. The $q$ exponent \eqref{eq:Qnumn} is $Q(m,n)=P(r,s)$ and $(-1)^{m+\ell}=(-1)^{r}$.  Furthermore, we have $m,n\ge 0$ or $m,n< 0$, if and only if, $r\ge s\ge -\ell$ or $r<s< -\ell$, respectively. Thus,
 the first sum of \eqref{eq:Ssum} gives
\[
\sum_{m,n \geq 0}  (-1)^{m+\ell}q^{Q(m,n)}=\sum_{s\ge -\ell}\,\sum_{r\ge s}(-1)^{r}q^{P(r,s)},
\] 
and the second sum gives
\begin{align*}
-\sum_{m,n < 0}  (-1)^{m+\ell}q^{Q(m,n)}&=\sum_{s<-\ell}\,\sum_{r< s}(-1)^{r+1}q^{P(r,s)}
=\sum_{s<-\ell}\,\sum_{r'\ge -s}(-1)^{r'}q^{P(r',s)},
\end{align*} 
where $r'=-r-1$ and using $P(r',s)=P(r,s)$. But $\sum_{r'= -s}^{s-1}(-1)^{r'}q^{P(r',s)}=0$ for $s<0$ since this sum is odd under the $r\rightarrow r'$ map. Thus 
\begin{align*}
-\sum_{m,n < 0}  (-1)^{m+\ell}q^{Q(m,n)}&=\sum_{s<-\ell}\,\sum_{r'\ge s }(-1)^{r'}q^{P(r',s)},
\end{align*}
and the result follows from  \eqref{eq:Ssum}. 
\end{proof}

Consider the rank $2$ indefinite theta function defined in \cite[Definition 2.1]{Zwegers-Thesis} for $Q$ and $B$ of \eqref{eqn:QABdef} and 
   $ a(\ell),b,c_{1},c_{2}$ of \eqref{eq:abc_data}.
Since $Q(\nu)\neq 0$ for $\nu\in\Z^{2}$  we find in this case that\footnote{This condition  is sufficient to ensure that $S_{Q}=\{\}$ (cf.\ \cite[Page 26]{Zwegers-Thesis}) leading to $\rho_{c_1,c_2}(\nu)$ as given.} 
\begin{align}
\theta_{a(\ell),b}^{c_{1},c_{2}}(\tau):=\sum_{\nu \in a(\ell)+\Z^{2}} \rho_{c_1,c_2}(\nu) e^{2\pi \im B(\nu ,b)} q^{Q(\nu)},
\notag
\end{align}	
where 
\begin{align*}
\rho_{c_1,c_2}(\nu) :=& \sgn B(\nu ,c_1)\left(1- \beta \left( -\tfrac{B\left(\nu, c_1\right)^{2}}{Q\left(c_1\right)}y\right)\right) 
-\sgn B(\nu ,c_{2})\left(1- \beta \left( -\tfrac{B\left(\nu, c_{2}\right)^{2}}{Q\left(c_{2}\right)}y\right)\right), 
\end{align*}
with $y=\operatorname{im}(\tau)$ and 	
where $\beta(x)=\int_{x}^{\infty}u^{-\half}e^{-\pi u}\, du$ for real $x\ge 0$ as before.
We define
\begin{align}
\notag
\vartheta^{\ell}(\tau):=\theta_{a(\ell),b}^{c_{1},c_{2}}(\tau).
\end{align}
\begin{corollary}
	\label{cor:thetaell}
 We have
 \begin{align*}
 \vartheta^{\ell}(\tau)&= 2 \im (-1)^{\ell } S_{\ell}(\tau) 
  + \sum_{j=1}^{2} (-1)^{j}  \sum_{\nu \in a(\ell)+\Z^{2}} 
  \sgn \left(B\left(\nu ,c_j\right)\right) \beta \left( -\tfrac{B\left(\nu, c_j\right)^{2}}{Q\left(c_j\right)}y\right) e^{2\pi \im B(\nu ,b)} q^{Q(\nu)}. 
 \end{align*}
 \hfill $\Box$
\end{corollary}
 The indefinite (and non-holomorphic) theta function $\vartheta^{\ell}$ satisfies the following remarkably simple modular properties.
 
\begin{lemma}
	\label{lem:thetaellmod}
	We have
\begin{align}
\vartheta^{\ell}(\tau+1) &= e^{ \im\pi\left(\frac{1}{4}-\frac{1}{2} \ell^{2} \right)}\vartheta^{\ell}(\tau),
\label{eq:Ttheta_ell}
\\
\vartheta^{\ell}\left(-\frac{1}{\tau}\right) &= -\frac{ \im\tau }{\sqrt{2}}
\sum_{\ell'=0,1\bmod 2}(-1)^{\ell\ell'}\vartheta^{\ell'}(\tau).
\label{eq:Stheta_ell}
\end{align}
\end{lemma}
\begin{proof}
Applying \cite[Corollary 2.9]{Zwegers-Thesis} we obtain
\[
\theta_{a,b}^{c_{1},c_{2}}(\tau+1)=
e^{-2\pi \im Q(a)-\pi \im B\left(A^{-1}A^{*},a\right)}
\theta_{a,a+b+\frac{1}{2} A^{-1}A^{*}}^{c_{1},c_{2}}(\tau),
\]
where $A^{*}=(1,2)$ is the vector of diagonal elements of $A$, and 
\begin{align*}
Q(a)=\frac{1}{4}\ell^{2}-\frac{1}{2} \ell+\frac{1}{8},\quad
B\left(A^{-1}A^{*},a\right)=-\ell+\frac{1}{2} ,\quad 
a+b+\frac{1}{2} A^{-1}A^{*}=\left( \frac{3}{2},-\frac{1}{2} \ell\right).
\end{align*}
Thus,
\[
e^{-2\pi \im Q(a)-\pi \im B\left(A^{-1}A^{*},a\right)}=-e^{ \im\pi\left(\frac{1}{4}-\frac{1}{2} \ell^{2} \right)}.
\]
Furthermore, \cite[Corollary 2.9]{Zwegers-Thesis} states that in general
\begin{align}
\theta_{a,b+\mu}^{c_{1},c_{2}}(\tau)=
e^{2\pi \im B(a,\mu)}\theta_{a,b}^{c_{1},c_{2}}(\tau),\quad \mu\in A^{-1}\Z^{2},
\notag
\end{align}
from which it follows in our case that $\theta_{a,a+b+\half A^{-1}A^{*}}^{c_{1},c_{2}}(\tau)=-\theta_{a,b}^{c_{1},c_{2}}(\tau)$. 
Thus \eqref{eq:Ttheta_ell} follows. Finally,  using \cite[Corollary 2.9]{Zwegers-Thesis} again we have in general that
 \[
 \theta_{a,b}^{c_{1},c_{2}}\left(-\frac{1}{\tau}\right)
 =
 \frac{\tau}{\sqrt{-\det A}}e^{2\pi \im B(a,b)}
 \sum_{p\in A^{-1}\Z^{2}\bmod \Z^{2}}\theta_{b+p,-a}^{c_{1},c_{2}}(\tau).
  \]
 In our case we find that $e^{2\pi \im B(a,b)}= \im (-1)^{\ell}$, $b+p=(\half ,-\half \ell')\bmod \Z^{2}$ for $\ell'=0,1$ and that 
 $\theta_{b+p,-a}^{c_{1},c_{2}}(\tau)
 =(-1)^{\ell\ell' +\ell+1}\theta_{a,b}^{c_{1},c_{2}}(\tau)$. This establishes \eqref{eq:Stheta_ell}.
\end{proof}

Thus  $S_{\ell}$ is given by the holomorphic part of the indefinite theta function $\vartheta^{\ell}$ and, as such, $(S_{0},S_{1})$ is a vector-valued mock modular form of weight $1$ for $\SL_2(\Z)$ (cf.\ \cite[Theorem 8.30]{BFOR-Harmonic}). Furthermore, since $\vartheta^{\ell}(\tau+4)=-\vartheta^{\ell}(\tau)$, it follows that $S_{\ell}$  is a mock modular form of weight $1$ for $\Gamma_{0}(4)$ with a  multiplier system. At this point, we have proved Proposition \ref{prop:MainProp}.

Alternatively to the explicit indefinite theta series completion shown above, we can use Zwegers' analysis of the more general Appell-Lerch sum $\mu (z_1,z_2;\tau)$ function as described in \cite{EguchiHikami-Mock}. This also allows us to find the completions for the functions $S_\ell$. 
 
Following the discussion surrounding \eqref{eqn:Bigmu} we have that $\theta_1(z; \tau)^2\mu (z; \tau)$ is a weight $\frac{3}{2}$ and index $1$ mock Jacobi form for $\Gamma^{J}$ with completion
\begin{align}
\notag
\theta_1 (z; \tau)^{2} \widehat{\mu }(z; \tau)
=\widehat{S}_{0}(\tau)\theta_{3}(2z; 2\tau)+\widehat{S}_{1}(\tau)\theta_{2}(2z; 2\tau),
\end{align}
using \eqref{eq:theta1_mu}, 
where $\widehat{S}_{\ell}$ is the completion of $S_{\ell}$ given by
\begin{align}
\label{eq:Shat1}
\widehat{S}_{0}(\tau)&:={S}_{0}(\tau)-  \theta_{2}(2\tau)C(\tau),
\\
\label{eq:Shat2}
\widehat{S}_{1}(\tau)&:={S}_{1}(\tau)+  \theta_{3}(2\tau)C(\tau).
\end{align}
It can be shown that the completion terms $-\theta_{2}(2\tau)C(\tau),  \theta_{3}(2\tau)C(\tau)$ agree with those of Corollary~\ref{cor:thetaell}. In the next section we find the functions $S_\ell$ and $C$ play roles in describing  $F^{\phi}$ of \eqref{eqn:Fdefn}.


\section{The mock modular form $F^{\phi}(\tau)$}\label{Section:FunctionF}

Define  column and row vectors $\Theta $ and  $H$ by
\begin{align*}
\Theta (z; \tau):=\begin{pmatrix}
\theta_{3}(2z; 2\tau)
\\
\theta_{2}(2z; 2\tau)
\end{pmatrix},
\quad 
H(\tau):=\left(f_{0}(\tau),\, f_{1}(\tau)\right),
\end{align*}
for $f_0,f_1$ of \eqref{eqn:EZdecomp}.
Then $\Theta $ determines a vector-valued Jacobi form of weight $\half$ and index $1$ for $\Gamma^{J}$ 
with a multiplier matrix\footnote{ 
Due to the $(c\tau+d)^{\frac{1}{2}}$ factor, $\rho_{\Theta}$ is a representation of the metaplectic double cover of $\SL_2(\Z)$.}
 $\rho_{\Theta}$. Indeed,  
\begin{align*}
\Theta (z; \tau+\lambda \tau+ \mu)&=q^{-\lambda^{2}}\zeta^{-2\lambda}\Theta (z; \tau),\quad \lambda,\mu \in \Z,
\\
\Theta \left(\frac{z}{c\tau+d}; \frac{a\tau+b}{c\tau+d}\right)
&=(c\tau+d)^{\half}e^{2\pi \im \frac{cz^{2}}{c\tau +d}}\rho_{\Theta}(\gamma)\Theta (z; \tau),\quad
\gamma=\begin{pmatrix}
a&b\\c&d
\end{pmatrix}\in \SL_2(\Z),
\end{align*}
where for $\SL_2(\Z)$ generators $T=\left(\begin{smallmatrix} 1&1\\0&1 \end{smallmatrix}\right)$ and $S=\left(\begin{smallmatrix} 0&-1\\1&0 \end{smallmatrix}\right)$ we have 
\begin{align}
\label{eq:rhotheta}
\rho_{\Theta}(T)=\begin{pmatrix}
1 & 0 \\ 0 & \im
\end{pmatrix},\quad 
\rho_{\Theta}(S)=\sqrt{\frac{- \im}{2}}\begin{pmatrix}
1 & 1 \\ 1 & -1
\end{pmatrix}.
\end{align}
This implies the $\Theta$ components are weight $\half$, index $1$, Jacobi forms for $\Gamma_{0}(4)^J$ (with $\Gamma_0(4)$ generated by $T$ and $ST^{4}S$)  with a multiplier system consisting of fourth roots of unity. 

 Given that $ \phi(z; \tau)=H(\tau)^{T}\Theta (z; \tau)$ is modular of weight $0$ for some level $M$ modular group $\Gamma_{\phi}\le \SL_2(\Z)$, it follows that $H$ determines a meromorphic vector-valued modular form of weight $-\half $ for $\Gamma_{\phi}$ with
\begin{align*}
H\left(\frac{a\tau+b}{c\tau+d}\right)
&=(c\tau+d)^{-\half}H(\tau)\rho_{\Theta}(\gamma)^{-1},\quad
\gamma=\begin{pmatrix}
a&b\\c&d
\end{pmatrix}\in\Gamma_{\phi}.
\end{align*} 
Define the modular group of level $\operatorname{lcm}(4,M)$ by
\begin{align}
\notag
\Gamma^{\phi}_{0}(4):= \Gamma_{\phi} \cap \Gamma_{0}(4)
=\left\{\begin{pmatrix}
a&b\\c&d
\end{pmatrix}\in \Gamma_{\phi} : c=0\bmod 4 \right\}.
\end{align}
It follows that $f_{0}(\tau)\theta_{3}(2\tau)$ and $f_{1}(\tau)\theta_{2}(2\tau)$ are modular invariant under 
$\Gamma^{\phi}_{0}(4)$ and are therefore either both constant (using \eqref{eq:f01chi}) or meromorphic. Thus $f_{0}, f_{1}$ are meromorphic modular forms with multipliers of weight $-\half $ for $\Gamma^{\phi}_{0}(4)$.

 Define the ``$R$-shift'' of $\phi$ (the $R$ relates to a notion of a Ramond character we define later) by
\begin{align}
\label{eq:Rchar}
R(\tau)=R^\phi (\tau) :=\phi\left(\half ;\tau  \right)
=f_{0}(\tau)\theta_{3}(2\tau)-f_{1} (\tau)\theta_{2}(2\tau),
\end{align}
which is $\Gamma^{\phi}_{0}(4)$ modular invariant. 
We may explicitly determine  $f_{0},f_{1}$  from   \eqref{eq:f01chi} and \eqref{eq:Rchar}. Altogether we have found the following. 

\begin{proposition}
	\label{prop:EZcoeff}
The theta coefficient functions $f_{0}, f_{1}$ for $\phi$ are given by
\begin{align}
\label{eq:hr}
f_{0}(\tau)=\frac{\chi+R(\tau)}{2\theta_{3}(2\tau)},\quad 
f_{1}(\tau)=\frac{\chi-R(\tau)}{2\theta_{2}(2\tau)},
\end{align}
for a constant $\chi = \phi (0; \tau)$ and where $R(\tau)=\phi\left(\half  ;\tau \right)$ is a modular invariant for $\Gamma^{\phi}_{0}(4)$ with singularities at rational cusps of $\Gamma^{\phi}_{0}(4)$.  The pair 
$(f_{0}, f_{1})$ determines a vector-valued modular form of weight $-\half$ for $\Gamma_{\phi}$ with multiplier matrix $\rho_{\Theta}^{-1}$ of \eqref{eq:rhotheta}. Moreover, the individual functions 
$f_{0}, f_{1}$ are meromorphic modular forms of weight $-\half $ with multipliers for $\Gamma^{\phi}_{0}(4)$, where $f_{0}(\tau)$ is singular at the singular cusps of $R(\tau)$ and the zeros of $\theta_{3}(2\tau)$, and $f_{1}(\tau)$ is singular at the singular cusps of $R(\tau)$ and the zeros of $\theta_{2}(2\tau)$. \hfill $\Box$
\end{proposition}

Altogether, from  the $\Gamma_{\phi}$ invariance of \eqref{eq:phi_f}, the discussion surrounding $C$ defined in \eqref{eqn:completionC}, and Proposition~\ref{prop:Fhr}, we  obtain the following.\footnote{We note that completion also follows from \eqref{eq:Shat1} and \eqref{eq:Shat2} together with \eqref{eq:f01chi}}

\begin{proposition}
\label{prop:hatFhr}
The series $F^{\phi}(\tau)$ is a mock modular form of weight $\half $ for some $\Gamma_{\phi}$ (with a multiplier system) with a universal completion 
\[
\widehat{F}^{\phi}(\tau)=F^{\phi}(\tau)+\chi C(\tau),
\]
where $\chi=\phi(0; \tau)$ is a constant.  \hfill $\Box$
\end{proposition}
Therefore we have established (ii) of Theorem~\ref{thm:MainTheoremF}.


\section{Jacobi trace functions of SVOAs and their modules}\label{Section:TraceFunctions}

We present a concept of parity root and discuss Jacobi (twisted) super trace functions of super vertex operator algebras (SVOAs) and investigate their properties. Also, we use the Jacobi twisted super trace to define the elliptic genus with parity root of such SVOAs.

\subsection{Jacobi trace functions}

Let $V$  be a SVOA of central charge $\cc$ with half integral Virasoro grading $V=\oplus_{n\in\half \Z}V_{n}$. Then $V$ also possesses a $\Z_2$ parity grading, $V=\Vbar{0}\oplus \Vbar{1}$, for parity $\vert v \rvert =2\wt(v)\bmod 1$ for $v\in V$ of Virasoro weight $\wt(v)$. 
Such a SVOA is often called a Nevue-Schwarz sector  in physics if it contains  a  superconformal or a fermionic subalgebra.
There is a canonical involution $\sigma$ associated with  parity given for $v$ of weight $\wt(v)$ by 
\[
\sigma v=(-1)^{|v|}v=e^{2\pi \im \wt(v)}v. 
\]

Let $h\in V_{1}$ be such that $h(0)$ acts semisimply with integer eigenvalues on a $V$-module $M$ and let  $h(1)h=2m\vac$ for ``level" or ``index" $m\in \C$. We call  such a vector $h$ a root of $V$. 
For the pair $(M,h)$, define the Jacobi trace function 
\begin{align}
Z_{M,h}(z; \tau)&:=\tr_{M}\left(\zeta^{h(0)}q^{L(0)-\tfrac{\cc}{24}}\right).
 \notag 
\end{align}

This is a natural extension of Jacobi trace functions for VOAs \cite{KrauelMason-Weak,KrauelMason-1pt}. 
Li's construction of a $g$-twisted module (see \cite{Li-Local} and below) based on the trivial automorphism $g=e^{2\pi \im h(0)}$ implies   for all $(\lambda,\mu) \in\Z^{2}$ that
\begin{align}
Z_{V,h}(z+\lambda  \tau+\mu ; \tau ) =q^{-m\lambda ^2}\zeta^{-2m\lambda }Z_{M_{\lambda},h}(z; \tau ),
\notag 
\end{align}
where $M_{\lambda}$ is a $\Z$-graded $V$-module \cite{DLiuMa,KrauelMason-Weak,KrauelMason-1pt}.

Now, we will study conditions that $Z_{M_{\lambda},h} =Z_{V,h}$ leading to Jacobi $z$ periodicity:
\begin{align}
Z_{V,h}(z+\lambda  \tau+\mu; \tau  ) =q^{-m\lambda ^2}\zeta^{-2m\lambda }Z_{V,h}(z; \tau ).
\label{eq:ZVJperiod}
\end{align}
Observe that \eqref{eq:ZVJperiod}  requires $m\in\half\Z$ since the Virasoro grading is half-integral.

\begin{lemma}
Assume that $Z_{V,h}(z; \tau)=\sum_{n\in \half \Z,r\in  \Z}c(n,r)q^{n-\cc/24} \zeta^r$ obeys \eqref{eq:ZVJperiod}. Then the coefficient $c(n,r)$ depends only on the discriminant $r^2-4mn$ and the value of $r \bmod 2m$:
\begin{align*}
c(n,r)=c_{r}(4mn-r^2),
\end{align*} 
where $c_{r'}(N)=c_r(N)$ for $r'=r\bmod 2m$ and $N\in\Z$.
\end{lemma}
\begin{proof}
The proof is essentially  that of \cite[Theorem~2.2]{EZ-JacobiForms} but where here we allow the $q$-exponents and index $m$ to be half-integral. 
\end{proof}
This lemma implies that $Z_{V,h}$ can be expanded in a basis of theta functions $\theta_{m,\mu}$ of \eqref{eq:EZtheta} as discussed in \cite[Section~5]{EZ-JacobiForms}. 
\begin{lemma}
	\label{lem:EZexp}
	Assume that $Z_{V,h}$ obeys \eqref{eq:ZVJperiod}. Then 
\begin{align}
Z_{V,h}(z; \tau)=\sum_{\mu \bmod 2m} g^{V,h}_\mu(\tau) \theta_{m,\mu}(z; \tau),
\label{eq:ZVEZ}
\end{align}
where for $\mu \bmod 2m$
\begin{align*}
g^{V,h}_\mu(\tau) :=q^{-\tfrac{\cc}{24}} \sum_{N\in\Z}c_{\mu}(N)q^{\frac{N}{4m}}.
\end{align*}
 \hfill $\Box$
\end{lemma}
We list some properties surrounding this decomposition.
\begin{corollary}
\label{cor:cpos}
The $q$-series $g^{V,h}_\mu$ obeys the following properties.
\begin{enumerate}
	\item[(i)] The coefficient $c_\mu(N)\ge 0$ for all $N\in\Z$.
	\item[(ii)] If the Virasoro grading is non-negative, then $c_\mu(N)=0$ for all $N<0$.
	\item[(iii)] We have $Z_{V,h}(z; \tau)$ is convergent  if and only if $g^{V,h}_\mu(\tau)$ is convergent for all $\mu \bmod 2m$ where  $\tau\in\HH$ and $z\in\C$.
\end{enumerate}
\end{corollary}
\begin{proof}
 For (i), since $Z_{V,h}$ is a trace function it follows that $c(n,r)\ge 0$. Turning to (ii) we note for non-negative grading the partition function can be written
\begin{align}
Z_{V,h}(0; \tau)&=\sum_{n\ge 0}\dim V_{n}\,q^{n-\tfrac{\cc}{24}}
=\sum_{\mu \bmod 2m} g^{V,h}_\mu(\tau)  \theta_{m,\mu}(0; \tau).
\label{eq:ZVpart}
\end{align}
 However, the coefficients of the $q$-series $g^{V,h}_\mu(\tau)$ and $  \theta_{m,\mu}(0; \tau)=\textthab{\frac{\mu}{2m}}{0}(2m\tau)$ are all non-negative. Hence it follows that $c_\mu(N)=0$ for all $N<0$. Finally, for (iii) if $g^{V,h}_\mu(\tau)$ is convergent for all $\mu \bmod 2m$ then $Z_{V,h}(z; \tau)$  is also from \eqref{eq:ZVEZ}. Conversely, from \eqref{eq:ZVpart}, the convergence of  $Z_{V,h}(0; \tau)$ implies the convergence of $g^{V,h}_\mu(\tau)$ for all $\mu\bmod 2m$ since all $q$-series have non-negative coefficients.
\end{proof}

In addition, we define a Jacobi super trace function by 
\begin{align}
\widetilde{Z}_{V,h}(z; \tau):=&
\operatorname{str}_{V}\left(\zeta^{h(0)}q^{L(0)-\tfrac{\cc}{24}}\right)
\notag\\
:=& \tr_{\Vbar{0}}\left(\zeta^{h(0)}q^{L(0)-\tfrac{\cc}{24}}\right)-\tr_{\Vbar{1}}\left(\zeta^{h(0)}q^{L(0)-\tfrac{\cc}{24}}\right)
\notag
\\
=&\tr_{V}\left(\sigma\zeta^{h(0)}q^{L(0)-\tfrac{\cc}{24}}\right).
\notag 
\end{align}
We call $h$  a \textit{parity root} of $V$ if $h$ generates the parity involution 
\begin{align}
\sigma=e^{ \im\pi h(0)}.
\notag 
\end{align}
Then we obtain a spectral flow relation for the supertrace
\begin{align}
\widetilde{Z}_{V,h}(z; \tau)&=Z_{V,h}\left(z+\half; \tau \right).
\label{eq:Zspec1}
\end{align}
Lemma \ref{lem:EZexp} along with \eqref{eq:EZthetaflow} implies the following.
\begin{lemma}
\label{LemmaI}
 We have 
 \[
 \widetilde{Z}_{V,h}(z; \tau) = \sum_{\mu \bmod{2m}} e^{\pi \im \mu} g_{\mu}^{V,h} (\tau) \theta_{m,\mu}(z; \tau).
 \]
  \hfill $\Box$
\end{lemma}

\subsection{A $\sigma$-twisted module - the  Ramond Sector}

Let us assume that a parity root $h\in V$ exists.
We construct a canonical $\sigma$-twisted $V$-module for $\sigma=e^{\pi \im h(0)}$ as follows. 
Define for all $v\in V$
\[
Y_{\sigma}(v,z):=Y(\Delta (\sigma, z)v,z), 
\quad \Delta(\sigma,z):=z^{\half h(0)}\exp \left(- \half\sum\limits_{n\geq 1}\frac{h(n)}{n}%
(-z)^{-n}\right) .
\]
Then $(V,Y_{\sigma})$ is a $\sigma$-twisted $V$-module by a theorem of Li \cite{Li-Local}.
This $\sigma$-twisted $V$-module is referred to as the Ramond sector in physics.
In particular we have
\begin{align*}
Y_{\sigma}(h,z) &=Y(h,z) + mz^{-1}\id_V,
\\
Y_{\sigma}(\omega,z) &=Y(\omega,z) +\half z^{-1}Y(h,z) 
+\tfrac{1}{4}mz^{-2}\id_V,
\end{align*}
using $h(1)h=2m\vac$ and 
where $\omega$ is the Virasoro vector.
With $Y_{\sigma}(v,z)=\sum_{n\in \Z+\half |v|}v_\sigma(n)z^{-n-1}$, we find that
\begin{align}
h_{\sigma}(0)=h(0)+m\id_V,
\notag 
\quad
L_{\sigma}(0)=L(0)+\half h(0)+\tfrac{1}{4}m\id_V.
\notag 
\end{align}
We define a Ramond sector $\sigma$-twisted Jacobi trace function by
\begin{align}
Z_{V_{\sigma},h}(z; \tau):=&\tr_{V}\left(\zeta^{h_\sigma(0)}q^{L_\sigma(0)-\tfrac{\cc}{24}}\right),
\notag 
\end{align} 
related to the original Jacobi trace function by a second spectral flow relation
\begin{equation}
\label{eq:Zspec2}
\begin{aligned}
Z_{V_{\sigma},h}(z; \tau)
&= q^{\frac{1}{4}m} \zeta^{m}\tr_{V}\left(\zeta^{h(0)}q^{L(0)+\half h(0)-\tfrac{\cc}{24}}\right)
\\
&=q^{\frac{1}{4}m}\zeta^{m}  Z_{V,h}\left(z+\half \tau; \tau \right).
\end{aligned} 
\end{equation}
Analogous to Lemma \ref{LemmaI} we obtain the next result.
\begin{lemma}
\label{LemmaII}
 We have 
 \[
 Z_{V_\sigma ,h}(z; \tau) = \sum_{\mu \bmod{2m}} g_{\mu}^{V ,h}(\tau) \theta_{m,\mu +m}(z; \tau).
 \]
  \hfill $\Box$
\end{lemma}

Finally, we define a Ramond sector  $\sigma$-twisted Jacobi supertrace function  by
\begin{equation}
\label{eq:SZVsig}
\begin{aligned}
\widetilde{Z}_{V_{\sigma},h}(z; \tau):=&\str_{V}\left(\zeta^{h_\sigma(0)}q^{L_\sigma(0)-\tfrac{\cc}{24}}\right)
\\
=&\tr_{V}\left(\sigma\zeta^{h_\sigma(0)}q^{L_\sigma(0)-\tfrac{\cc}{24}}\right),
\end{aligned}
\end{equation} 
for $\sigma = e^{\pi \im h(0)}$ related to the $\sigma$-twisted Jacobi trace function by a third spectral flow relation
\begin{equation}
\notag 
\begin{aligned}
\widetilde{Z}_{V_{\sigma},h}(z; \tau)&=(-1)^{m} Z_{V_{\sigma},h}\left(z+\half ; \tau\right)
=q^{\frac{1}{4}m} \zeta^{m} Z_{V,h}\left(z+\half \tau+\half ; \tau\right).
\end{aligned}
\end{equation} 

We define the \textit{elliptic genus with parity root $h$} of the SVOA $V$ to be the Ramond supertrace
\begin{align}
\mathcal{E}_{V,h}(z; \tau):=\widetilde{Z}_{V_{\sigma},h}(z; \tau).
\notag 
\end{align}
\begin{remark}
	Note that $h$ is a parity root if and only if $L(0)+\half h(0)$ has integral eigenvalues.
	Thus $L_{\sigma}(0)-\cc/24$ has integral eigenvalues when $\cc=6m$ as is found in the $N=4$ literature.
\end{remark}
\begin{remark}
\label{Remark:Ramond}
The Ramond sector character $\operatorname{tr}_{V_{\sigma ,h}}(q^{L_\sigma(0)-\cc/24})$ equals $(-1)^m R_{V_{\sigma},h}(\tau)$ for 
$R_{V_{\sigma},h}(\tau):=\mathcal{E}_{V,h}(\half ;\tau)$ (cf.\ \eqref{eq:Rchar}).
\end{remark}
Again, using \eqref{eq:EZthetaflow} we obtain the decomposition of $\mathcal{E}_{V,h}$.
\begin{lemma}
\label{LemmaIII}
 We have
 \[
 \mathcal{E}_{V,h}(z; \tau) = \sum_{\mu = \bmod{2m}} e^{\pi \im \mu} g_{\mu}^{V,h}(\tau) \theta_{\mu ,m+\mu} (z; \tau).
 \]
  \hfill $\Box$
\end{lemma}


\section{Applications to the elliptic genus of a $K3$ surface}\label{Section:K3}

Before getting to our main example in the next section, we pause to discuss the theory developed so far applied to the elliptic genus of a $K3$ surface. The bulk of the content of this section is known in the literature since the $K3$ elliptic genus is a weak Jacobi form on the full Jacobi group, and thus falls in the realm of, for instance \cite{ EOT-Mathieu, CDH-Umbral, CDH-UmbralNiemeier} and others. However, it is worth seeing how the material presented above can be used to quickly recover what is known.
 
\subsection{Calculating $F^{K3}$ from the elliptic genus}

For the $K3$ elliptic genus we have $\phi^{K3}(z; \tau)=2\phi_{0,1}(z; \tau)$
 with 
 \begin{align*}
2\phi_{0,1}(z; \tau)=& 
8\left(
\frac{\theta_{2}(z; \tau)^{2}}{\theta_{2}(\tau)^{2}}
+\frac{\theta_{3}(z; \tau)^{2}}{\theta_{3}(\tau)^{2}}
+\frac{\theta_{4}(z; \tau)^{2}}{\theta_{4}(\tau)^{2}}
\right),
\end{align*}
a weak Jacobi form of weight $0$ and index $1$ for $\SL_2(\Z)$. 
Then we find
\begin{align*}
2\phi_{0,1}\left(z+\half ; \tau\right)=& 
8\left(
\frac{\theta_{1}(z; \tau)^{2}}{\theta_{2}(\tau)^{2}}
+\frac{\theta_{4}(z; \tau)^{2}}{\theta_{3}(\tau)^{2}}
+\frac{\theta_{3}(z; \tau)^{2}}{\theta_{4}(\tau)^{2}}
\right),
\end{align*}
so that from \eqref{eq:Rchar} we obtain
\begin{align*}
R^{K3}(\tau):= R^{\phi^{K3}}(\tau)=& 
8\left(
\frac{\theta_{4}(\tau)^{2}}{\theta_{3}(\tau)^{2}}
+\frac{\theta_{3}(\tau)^{2}}{\theta_{4}(\tau)^{2}}
\right)
\\
=& 16+512 q+4096 {q}^{2}+22528 {q}^{3}+98304 {q}^{4}+367616 {q}^{5}+\ldots 
\end{align*}
One can check that $R^{K3}(\tau)$ is $\Gamma_{0}(4)$ modular invariant as expected from Proposition~\ref{prop:EZcoeff}. There are three independent equivalence classes of cusps for $\Gamma_{0}(4)$ for  which we can choose representatives, $0$, $\half$, and $\infty$. Additionally, $\Gamma_{0}(4)$ is a genus zero modular group, and in fact  $R^{K3}(\tau)$ is a hauptmodul with a pole of order one at $0$ so that
\begin{align*}
R^{K3}(\tau)
=-16+32\left(\frac{\theta_{3}(2\tau)}{\theta_{4}(2\tau)}\right)^{4}=16+512\left(\frac{\eta(4\tau)}{\eta(\tau)}\right)^{8}.
\end{align*} 
Since $\chi=24$ we obtain the theta coefficients $f_{0}^{K3},f_{1}^{K3}$ via Proposition \ref{prop:EZcoeff} to be 
\begin{equation}
\label{eqn:K3h}
\begin{aligned}
f_{0}^{K3} =&
\frac{4\,\theta_{4}(2\tau)^{4}+16\,\theta_{3}(2\tau)^{4}}{\theta_{3}(2\tau)\theta_{4}(2\tau)^{4}}
=
20+216 q+1616 {q}^{2}+8032 {q}^{3}
+\ldots ,\\
f_{1}^{K3} =&
\frac{4\,\theta_{4}(2\tau)^{4}-16\,\theta_{2}(2\tau)^{4}}{\theta_{2}(2\tau)\theta_{4}(2\tau)^{4}}
=
-q^{-\frac{1}{4}}\left(-2+128 q+1026 {q}^{2}+5504 {q}^{3}
+\ldots\right),
\end{aligned}
\end{equation}
(using $\theta_{2}^{4}-\theta_{3}^{4}+\theta_{4}^{4}=0$). These describe a vector-valued modular form of weight $-\half $ for $\SL_2 (\Z)$ with multiplier matrix $\rho_{\Theta}^{-1}$. By Proposition \ref{prop:Fhr} we find, as expected, that
\begin{equation}
\notag
F^{K3}(\tau)=q^{-\frac{1}{8}}\left(-2+90 q+462 {q}^{2}+1540 {q}^{3}+4554 {q}^{4}+11592 {q}^{5}+\ldots \right).
\end{equation}

\subsection{Parity roots and the elliptic genus of a $K3$ surface}

At this point, we have only analyzed the elliptic genus of a $K3$ surface using the theory developed in Sections \ref{Section:ChiExp}--\ref{Section:FunctionF}. However, one could ask whether the parity root setup of a SVOA in Section \ref{Section:TraceFunctions} could give rise to the elliptic genus of a $K3$ surface discussed in the previous subsection.

Let $V_{K3}$ denote a SVOA of central charge $\cc$ with root $h$ such that 
\[
\mathcal{E}_{V_{K3},h}(z; \tau)=\pm 2\phi_{0,1}(z; \tau).
\]
The following proposition shows that for this to occur, $h$ cannot be a parity root.

\begin{proposition}
We have $h\in V_{K3}$ is not a parity root.
\end{proposition}

\begin{proof} Suppose that $h$ is a parity root. 
By Lemma \ref{LemmaIII} we find
\[
 \mathcal{E}_{V_{K3},h}(z; \tau) = g_0^{K3,h}(\tau)\theta_{1,1}(z; \tau)-g_1^{K3,h}(\tau)\theta_{1,0}(z; \tau).
\]
Comparing to \eqref{eqn:EZdecomp} implies $g_0=\pm f_{1} $ and $g_1=\mp f_{0} $ of \eqref{eqn:K3h}. However, this contradicts Corollary~\ref{cor:cpos}~(i) since the coefficients of $g_{0}^{K3,h}(\tau)$ are of mixed signature. Hence $h$ is not a parity root.
\end{proof}

This rules out some ways in which the elliptic genus of a $K3$ surface may arise from a SVOA.

\begin{corollary}
The elliptic genus of a $K3$ surface cannot be obtained as the  elliptic genus of  a SVOA with a parity root $h$.  \hfill $\Box$
\end{corollary}


\section{A rank $6$ SVOA example}\label{Section:Rank6}

\subsection{The structure}

In this subsection, we first recall the connection between a particular odd lattice SVOA example and the $N=4$ superconformal algebra $\mathcal{A}$ of central charge $6$. We refer the reader to \cite{Kac-Beginners} for more information about general SVOA theory.

Let $X=\{\alpha_{i}\}_{i=1}^{6}$  be an orthogonal basis for $\R^6$ with  $(\alpha_{i},\alpha_{j})=3\delta_{ij}$ for $i,j=1,\ldots,6$. 
Let ${L+}$  be the  odd, positive-definite, integral lattice spanned by $X$ together with
\begin{align*}
h{:=}\tfrac{1}{3}(\alpha_1{+}\hdots{+}\alpha_6){\in} {L+}. 
\end{align*}
  For ease of notation, throughout the remainder of this section we set $L:={L+}$. 
Note that $(h,\alpha_{i})=1$ for each $i$ and $(h,h)=2$. 

It was shown in \cite[Example 2 and Proposition 26]{MTY-N=4} that the lattice SVOA $V_{L}=\oplus_{n=0}^{\infty}(V_{L})_{(n)}$ contains an $N=4$
superconformal algebra of central charge $6$ generated by the following four states of weight $\tfrac{3}{2}$:
\begin{equation}
 \notag 
\begin{aligned}
&\tau^{+}{:=}\sum_{\alpha{\in}X}c_{\alpha}e^{\alpha},\quad  \tau^{-}{:=}\sum_{\alpha{\in}X} c_{\alpha}\varepsilon(h, \alpha)e^{\alpha-h},
\\ &\overline{\tau}^{+}{:=}\sum_{\alpha\in X} d_{h{-}\alpha}e^{h-\alpha},\quad
 \overline{\tau}^{-}{:=}-\sum_{\alpha\in X} d_{h-\alpha}\varepsilon(h, \alpha)e^{-\alpha},
\end{aligned}
\end{equation}
where  
\begin{equation}
\notag 
c_{\alpha}d_{h-\alpha}\varepsilon(h,  \alpha)=-\tfrac{1}{3},
\end{equation}
for lattice cocycle factor  $\varepsilon$. In addition, as Lie algebras, $(V_{L})_{1}$ and $\slt$ are isomorphic \cite{MTY-N=4}. It is known that $V_L$ is a subVOA of nearly all odd Niemeier lattice SVOAs  \cite{HoehnMason-Most}.

\subsection{The associated trace functions}

We utilize the theory developed in the sections above to compute $F^L :=F^{\mathcal{E}_{V_{L},h}}$ for the function $\mathcal{E}_{V_{L},h}$.

\subsubsection{Jacobi trace functions for $\VLp$}

Let $L_3=\langle \alpha_{i} \rangle\cong (\sqrt{3}\Z)^6$, a rank 6 sublattice of $L$.
Clearly $3h\in L_3$ so we obtain the coset decomposition
\begin{align}
L=L_3\cup (L_3+h)\cup( L_3+2h).
\notag 
\end{align}
We may write $\beta\in L$ as $\beta=\sum_{i=1}^{6}\left(m_i+\frac{r}{3}\right) \alpha_i$ for $m_i\in\Z$ and $r\in\{0,1,2\}$.
Therefore,
\begin{equation}
\notag 
\begin{aligned}
(\beta,\beta)&=3\sum_{i=1}^6 \left(m_i +\frac{r}{3}\right)^2
=\sum_{i=1}^6 (3m_i ^2+2r m_i)+2r^2,
\\
(\beta,h)  &=\sum_{i=1}^6 \left(m_i +\frac{r}{3}\right)
= \sum_{i=1}^6 m_i +2r.
\end{aligned}
\end{equation}
We also note the identity
\begin{align}
(\beta,\beta)=(\beta,h) \bmod 2.
\label{eq:bbh}
\end{align}
A straightforward calculation shows that for $r=0,1,2$, we have (cf.\ \eqref{eq:thetaLh})
\begin{align*}
\theta^h_{L_3+rh}(z; \tau):=\sum_{\beta\in L_3+rh} q^{\half(\beta,\beta)}\zeta^{(\beta,h)}=\left(\thab{\frac{r}{3}}{0}(z; 3\tau)\right)^6
\end{align*}
and therefore the Jacobi trace function for $V_L$ and $h$ is given by
\begin{align}
Z_{\VLp,h}(z; \tau)&=\frac{\theta_{L}^h(z; \tau)}{\eta(\tau)^6}
= \sum_{r\in\{0,1,2\}}\left( \eta (\tau)^{-1} \thab{\frac{r}{3}}{0}(z; 3\tau)\right)^6.
\label{eq:ZLp}
\end{align}
We may also express $Z_{\VLp,h}$ in terms of its theta decomposition \eqref{eq:ZVEZ} to find
\begin{align}
Z_{\VLp,h}(z; \tau)=g^{\VLp,h}_0(\tau) \theta_{1,0}(z; \tau)+g^{\VLp,h}_1(\tau) \theta_{1,1}(z; \tau)
\notag 
\end{align}
(recall \eqref{eqn:ThetaRelation}), and thus
\begin{align*}
g^{\VLp,h}_0(\tau) &= q^{-\frac{1}{4}}\left(1+6 q+57 {q}^{2}+308 {q}^{3}+1305 {q}^{4}+4800 {q}^{5}+15764 {q}^{6} \right.
\\
&
\qquad \left. +47466 {q}^{7}+133461 {q}^{8}+\ldots
  \right),
  \\
g^{\VLp,h}_1(\tau) &= 12\,q+92\,{q}^{2}+444\,{q}^{3}+1836\,{q}^{4}+6520\,{q}^{5}+20916\,{q}^{6}
\\
&\qquad +61824\,{q}^{7}+171244\,{q}^{8} +\ldots .
\end{align*}
Note that all the coefficients are non negative integers as expected from Corollary~\ref{cor:cpos}. Closed formulas for $g^{\VLp,h}_0$ and $g^{\VLp,h}_1$ follow from Proposition~\ref{prop:EZcoeff} as described below.

Let $v=u\otimes e^{\beta}\in  \VLp$ with $\wt(v)\in\Z +\half (\beta,\beta)$. 
Then, using  \eqref{eq:bbh} we find $$\sigma v=e^{ \im\pi(\beta,\beta)}v=e^{ \im\pi(\beta,h)}v=e^{ \im\pi h(0)}v.$$ Thus $h$ is a parity root of $\VLp$. Hence  from \eqref{eq:thetaflow} and \eqref{eq:Zspec1} the Jacobi supertrace function is given by
\begin{align}
\widetilde{Z}_{\VLp,h}(z; \tau)&={Z}_{\VLp,h}\left(z+\half; \tau \right)
=\sum_{r\in\{0,1,2\}}\left(\eta (\tau)^{-1} \thab{\frac{r}{3}}{\half}(z; 3\tau)\right)^6
\notag 
\\
&=g^{\VLp,h}_0(\tau) \theta_{1,0}(z; \tau)-g^{\VLp,h}_1(\tau) \theta_{1,1}(z; \tau),
\notag \label{eq:SZLp}
\end{align}
using Lemma \ref{LemmaI}.

\subsubsection{Jacobi $\sigma$-twisted (super)trace functions for $\VLp$}
The parity root $h\in\VLp$ is of level $m=1$ so that by \eqref{eq:thetaflow}, \eqref{eq:Zspec2} and \eqref{eq:ZLp} we find
\begin{equation}
\notag \label{eq:ZLpsig}
\begin{aligned}
Z_{\left(\VLp\right)_\sigma,h}(z; \tau)&=
 q^{\frac{1}{4}}\zeta\, \frac{\theta_{L}^h\left(z+\half \tau,\tau\right)}{\eta(\tau)^6}
= \sum_{r\in\{0,1,2\}}\left(\eta (\tau)^{-1} \thab{\frac{r}{3}+\frac{1}{6}}{0}(z; 3\tau )\right)^6
\\& = \sum_{r\in\{1,3,5\}}\left(\eta (\tau)^{-1}\thab{\frac{r}{6}}{0}(z; 3\tau )\right)^6
=\frac{\theta_{L+\half h}^h\left(z; \tau\right)}{\eta(\tau)^6}.
\end{aligned}
\end{equation}
This last expression also follows from the construction  of the $e^{ \im\pi h(0)}$-twisted module  from  the coset lattice $L+\half h$. In terms of its theta decomposition we find  from Lemma \ref{LemmaII} that
\begin{align}
Z_{\left(\VLp\right)_\sigma,h}(z; \tau)=g^{\VLp,h}_0(\tau) \theta_{1,1}(z; \tau)+g^{\VLp,h}_1(\tau) \theta_{1,0}(z; \tau).
\notag 
\end{align}
Lastly, we compute the elliptic genus \eqref{eq:SZVsig} to find
\begin{align}
\mathcal{E}_{\VLp,h}(z; \tau)&=-{Z}_{\left(\VLp\right)_\sigma,h}\left(z+\half; \tau \right)
=-\sum_{r\in\{1,3,5\}}\left(\eta (\tau)^{-1} \thab{\frac{r}{6}}{\half}(z; 3\tau )\right)^6
\notag 
	\\ \notag
	&= 
 \zeta^{-1}+\zeta + \left( -12 +6 \left(\zeta^{-1}+\zeta\right) \right) q
 \\ \notag
 &\qquad + \left( -92+57\left( \zeta^{-1}+ \zeta\right) -12 \left(\zeta^{-2}+\zeta ^{2}\right)+\zeta^{-3}+\zeta^{3}\right) {q}^{2}
 \\ \notag
 &\qquad + \left(-444 +308 \left(\zeta^{-1}+ \zeta\right) 
-92\left( \zeta^{-2}+\zeta^{2}\right) +6 \left(\zeta^{-3}+ \zeta^{3}\right) \right) {q}^{3}+\ldots .
\end{align}
In particular, we obtain the following result.

\begin{lemma}
The function $\mathcal{E}_{\VLp,h}$ is a weak Jacobi form of weight $0$ and index $1$ for $\Gamma_{0}(3)^J$.
\end{lemma}

\begin{proof}
The modular subgroup $\Gamma_{0}(3)$ is generated by $T,ST^{3}S$ so we need only check the  $ST^{3}S$ modular transformation property. For $\gamma=\left(\begin{smallmatrix} a&b \\ c& d \end{smallmatrix}\right)\in \SL_2(\Z)$ and $(z; \tau)\in \mathbb{C} \times \mathbb{H}$ set $\gamma.\tau := \tfrac{a\tau +b}{c\tau +d}$ and $\gamma.z=\tfrac{z}{c\tau+d}$. Using \eqref{eqn:ThetaTransforms} we find 
\begin{align*}
\mathcal{E}_{\VLp,h}\left(S.z; S.\tau\right)
=&
\tfrac{1}{27}e^{2\pi \im \tfrac{z^{2}}{\tau}} \sum_{s\in\{1,3,5\}}\left(
\eta (\tau)^{-1}
\thab{\half}{-\frac{s}{6}}\left(\frac{z}{3}; \frac{\tau}{3} \right)
\right)^{6}
\end{align*}
so that
\begin{align*}
\mathcal{E}_{\VLp,h}\left(T^3S.z; T^3S.\tau\right) =&
\tfrac{1}{27}e^{2\pi \im \tfrac{z^{2}}{(\tau+3)}} \sum_{s\in\{1,3,5\}}\left(
\eta (\tau)^{-1}
\thab{\half}{-\frac{s}{6}}\left(\frac{z}{3}; \frac{\tau}{3}\right)
\right)^{6}.
\end{align*}
This implies 
\[
\mathcal{E}_{\VLp,h}\left( S T^{3}S.z; S T^{3}S.\tau \right)=e^{2\pi \im \tfrac{3z^{2}}{(3\tau-1)}}\mathcal{E}_{\VLp,h}\left(z; \tau\right).
\]
The elliptic transformation is straight-forward.
\end{proof}
By Lemma \ref{LemmaIII} and \eqref{eqn:EZdecomp} we find $f_{1} ^{V_l,h}=g_0^{V_L,h}$ and $f_{0} ^{V_L,h}=-g_1^{V_L,h}$. These may be obtained from Proposition \ref{prop:EZcoeff} in terms of $R^L$ and $\chi^L$. We have $R^{L}(\tau):=\mathcal{E}_{\VLp,h}\left(\half ;\tau \right)$ is given by
	\begin{align*}
	R^{L}(\tau)=& 
-\sum_{s\in\{1,3,5\}} \left(\frac{1}{\eta(\tau)}\thab{\frac{s}{6}}{0}(0; 3\tau )\right)^{6}
	\\
	=& -\left(2+24 q+232 {q}^{2}+1256 {q}^{3}+5448 {q}^{4}+20432 {q}^{5}+\ldots \right) ,
	\end{align*} 
 and is $\Gamma_{0}(12)$ invariant by Proposition~\ref{prop:EZcoeff}. Note that Remark \ref{Remark:Ramond} implies $-R^L(\tau)=\operatorname{tr}_{V_{\sigma,h}} (q^{L_\sigma(0)-\cc/24})$ is the Ramond character with non-negative coefficients as expected.
\begin{lemma}
We have
$\chi^L=\mathcal{E}_{\VLp,h}(0; \tau)=2$.
\end{lemma}

\begin{proof}
Euler's pentagonal number theorem states that 
\begin{align*}
\eta(\tau)&=q^{\frac{1}{24}}\sum_{n\in\Z} (-1)^n q^{\frac{1}{2}n(3n+1)}
=e^{-\frac{ \im\pi}{6}}
\thab{\frac{1}{6}}{\half}(3{\tau} ).
\end{align*}
Using \eqref{eq:minus} we also have
\begin{align}
\thab{-\frac{1}{6}}{\half}(3\tau )
&=
\thab{\frac{1}{6}}{-\half}(3\tau )
=
e^{-\frac{i \pi}{3}}
\thab{\frac{1}{6}}{\half}(3\tau )=e^{-\frac{i \pi}{6}}\eta(\tau).
\notag 
\end{align}
Together with \eqref{eq:aplus1} this gives
\begin{align*}
\left(\thab{\frac{1}{6}}{\half}(3{\tau})\right)^6=\left(\thab{\frac{5}{6}}{\half}(3{\tau})\right)^6=-\eta(\tau)^6.
\end{align*}
Meanwhile, $\textthab{\half}{\half}( 0; \tau)=0$ implies $\mathcal{E}_{\VLp}(0; \tau)=2$.
\end{proof}
Proposition~\ref{prop:EZcoeff} implies we may find $f_{0} ^L$, $f_{1} ^L$ from \eqref{eq:hr} in terms of $R^{L}$ and $\chi^{L}$ to obtain a vector-valued modular form for $\Gamma_0(3)$ of weight $-\half$ with multiplier matrix $\rho_{\Theta}^{-1}$.
Finally, using Proposition \ref{prop:Fhr} the mock modular form \eqref{eqn:Fdefn} discussed in the introduction, $F^L = F^{\mathcal{E}_{\VLp,h}}$, can easily be obtained.
\begin{lemma}
We have
	\begin{align*}
	F^{L}(\tau)&= f_{0}^L(\tau)S_{1}(\tau)-f_{1}^L(\tau)S_{0}(\tau)
	\\&=q^{-\frac{1}{8}}\left(1+5 q+29 {q}^{2}+80 {q}^{3}+253 {q}^{4}+654 {q}^{5}+\ldots \right).
	\end{align*}
	\hfill $\Box$
\end{lemma}



\end{document}